\documentclass[12pt]{amsart}

\usepackage{amssymb}

\usepackage{enumitem}

\usepackage{graphicx}

\makeatletter
\@namedef{subjclassname@2020}{%
  \textup{2020} Mathematics Subject Classification}
\makeatother

\usepackage[T1]{fontenc}

\newtheorem{theorem}{Theorem}[section]
\newtheorem{corollary}[theorem]{Corollary}
\newtheorem{lemma}[theorem]{Lemma}

\theoremstyle{definition}
\newtheorem{definition}[theorem]{Definition}
\newtheorem{remark}[theorem]{Remark}

\numberwithin{equation}{section}

\newtheorem{question}[theorem]{Question}
\newtheorem{conjecture}[theorem]{Conjecture}

\newcommand{\seq}[1]{\langle #1\rangle}
\newcommand{\ns}{\varnothing}

\newcommand{\cont}{\mathcal C}
\newcommand{\homeo}{\mathcal H}
\newcommand{\surj}{\mathcal S}
\newcommand{\shad}{\mathcal{T}}

\newcommand{\N}{\omega}

\begin{document}


\baselineskip=17pt


\title[Sample paper]{A sample paper for IMPAN journals}

\title[Shadowing in One-Dimension]{On Genericity of Shadowing in One Dimension}

\author[J. Meddaugh]{Jonathan Meddaugh}
\address[J. Meddaugh]{Department of Mathematics, Baylor University, Waco TX, 76798}
\email[J. Meddaugh]{Jonathan\_Meddaugh@baylor.edu}

\date{}

\begin{abstract} We show that shadowing is a generic property among continuous maps and surjections on a large class of locally connected one-dimensional dynamical systems.
\end{abstract}


\maketitle

\section{Introduction} One of the most well-studied ideas in topological dynamics is that of the \emph{stability} of a dynamical system. There are many appropriate notions of stability--hyperbolicity, Lyapunov stability, topological stability, structural stability, among many others. Several of these notions have connections to the \emph{pseudo-orbit tracing property} (\emph{shadowing property}) of dynamical systems \cite{Pilyugin-LipStruct,Slackov,WaltersFiniteType}. The shadowing property was initially studied by Anosov \cite{Anosov} and Bowen \cite{Bowen} and it (along with its variations) has been a very active area of study since \cite{Brian-Ramsey,BMR-Variations,Coven,Fakhari,GrebogiYorke,sakai,Mai,MR,Ombach2,Oprocha-Thick,palmer,Pilyugin-Orbit,shimomura,WaltersFiniteType}.

Informally, systems with shadowing have the property that the behaviors witnessed by the \emph{pseudo-orbits} of a system (i.e. orbits with some allowed amount of error) are representative of true behaviors of the system in the sense that there every pseudo-orbit has an \emph{orbit} which shadows it. Since pseudo-orbits arise naturally in computed orbits of dynamical systems, shadowing is a very useful property in computer-assisted dynamics, as we are guaranteed that computer-generated orbits are indeed representative of actual orbits for the system. As such, classifying systems which have shadowing is quite useful. In a variety of contexts, full and partial characterizations exist. In particular, in shift spaces \cite{WaltersFiniteType} and in the class of tent maps on the interval \cite{Coven} the shadowing property has been completely characterized. Partial characterizations exist in other contexts, including the action of a quadratic polynomial on its Julia set \cite{BMR-ToAppear, BR-ToAppear}. However, in more general classes of dynamical systems, the characterization problem is intractable.

It is thus reasonable to ask whether a system on a given topological space might be expected to have shadowing. More specifically, given a compact metric space $X$, let $\cont(X)$ denote the space of continuous self-maps on $X$ with topology induced by the supremum metric. Let $\shad(X)$ denote the subset of $\cont(X)$ consisting of those maps with shadowing. We can then re-frame the question of whether a dynamical system might be expected to have shadowing by instead asking whether $\shad(X)$ is a \emph{generic} subset of $\cont(X)$, i.e. whether $\shad(X)$ contains a dense $G_\delta$ subset of $\cont(X)$. Determining the genericity of shadowing in the classes $\surj(X)$ and $\homeo(X)$ of surjections and of homeomorphisms, respectively is also useful.

The question of genericity of shadowing has been examined for some time in a variety of specific contexts. In particular, many genericity results exist for manifolds. Yano demonstrated that shadowing is generic among homeomorphisms of the unit circle \cite{Yano}, Odani extended this result to smooth manifolds of dimension at most three \cite{Odani}, and Pilyugin and Plamanevskaya  further extended this result to compact manifolds without boundary which have a handle decomposition \cite{Pilyugin-Plam}. Results concerning the more general class of continuous maps include those of Mizera, who demonstrated that shadowing is generic in $\cont(X)$ where $X$ is an arc or circle \cite{Mizera}, and those of Ko\'scielniak, Mazur, Oprocha and Pilarzyk, who extended this result to $\cont(X)$ and $\surj(X)$ where $X$ is a compact manifold \cite{Mazur-Oprocha}. 

For less homogeneous spaces $X$, far less is known. Recently, it has been shown that these results extend to the class of continuous maps on dendrites, i.e.
locally connected, uniquely arc-wise connected continua \cite{BMR-Dendrites}. This is a class of spaces which include acyclic graphs and the Julia sets of a large class of quadratic polynomials \cite{Car-Gam}. Shadowing is also known to be generic in the class of chainable continua which retract onto subarcs in a reasonable way \cite{KMPK}. The main goal of this paper is to extend this result to a yet larger class of locally connected continua. We define the class of \emph{graphites} which consist of those locally connected continua which appropriately retract onto subgraphs. This larger class of continua includes the Menger curve, the Sierpinski carpet and the Sierpinski gasket. We then demonstrate that shadowing is a generic property in the classes of continuous maps and surjections on graphites.

The structure of the paper is as follows. In Section \ref{Prelim}, we introduce terminology and prove that graphites have particularly nice open covers. In Section \ref{maps}, we use a construction similar to that of \cite{BMR-Dendrites} to demonstrate that shadowing is generic on graphites. In Section \ref{surj}, we modify this construction to achieve the same result in the space of surjections on graphites. In Section \ref{conc} we pose some questions for future research.

\section{Preliminaries}\label{Prelim}

For a compact metric space $(X,d)$, let $\cont(X)$ denote the space of all continuous maps $f:X\to X$ endowed with the metric given by 
\[\rho(f,g)=\max_{x\in X}d(f(x),g(x)).\] 
Furthermore, let $\surj(X)$ denote the subset of $\cont(X)$ consisting of surjective maps.

Let $f:X\to X$ be a continuous map. For a fixed $\delta>0$, a \emph{$\delta$-pseudo-orbit for $f$} is a sequence $\seq{x_i}_{i\in\N}$ such that $d(x_{i+1},f(x_i))<\delta$ for all $i\in\N$. A sequence $\seq{x_i}$ which is a $\delta$-pseudo-orbit for all $\delta>0$ is called an \emph{orbit}, and it is immediate that there is some $z\in X$ with $x_i=f^i(z)$ for all $i\in\N$.

The map $f:X\to X$ has \emph{shadowing} provided that for every $\epsilon>0$ there exists $\delta>0$ such that for each $\delta$-pseudo-orbit $\seq{x_i}_{i\in\N}$ there exists an orbit $\seq{f^i(z)}_{i\in\N}$ such that $d(x_i,f^i(z))<\epsilon$ for all $i\in\N$; in this case, we say that the pseudo-orbit is $\epsilon$-shadowed by the orbit of $z$. Let $\shad(X)$ denote the subspace of $\cont(X)$ consisting of those maps with shadowing.

The following auxiliary notion will be useful in our construction.
\begin{definition}
	Fix an open cover $\mathcal U=\{U_1,\ldots U_k\}$  of $X$ and a map $h:X\to X$. Define the pattern of $h$ with respect to $\mathcal U$ to be $\phi_{\mathcal U,h}:\{1,\ldots k\}\to 2^{\{1,\ldots k\}}$ as follows. For $i=1,\ldots k$, and $h:X\to X$, let  \[\phi_{\mathcal U,h}(i)=\{j:h({U_i})\cap U_j\neq\ns)\}.\]
	When there is no ambiguity as to which cover is being used, we will simply denote this by $\phi_h$.
\end{definition}

In the sense of \cite{good-meddaugh}, sequences $\seq{j_i}_{i\in\omega}$ in $\{1,\ldots, k\}^\omega$ which satisfy $j_{i+1}\in\phi_{\mathcal U,h}(j_i)$ are the $\mathcal U$-pseudo-orbit patterns for $h$. For the purposes of this paper, it is enough to note that for sufficiently small $\delta$, each $\delta$-pseudo-orbit $\seq{x_i}$ of $h$ has a $\mathcal U$-pseudo-orbit pattern $\seq{U_{j_i}}$ with $x_i\in U_{j_i}$. 


A space $X$ is \emph{path-connected} if for any $x,y\in X$ there exists a continuous map $f:[0,1]\to X$ with $f(0)=x$ and $f(1)=y$. Note that for compact metric spaces, this is equivalent to arc-connectedness, and so the map $f$ can be taken to be an injection. For the purposes of this paper, a \emph{graph} is a compact connected metric space $G$ for which there exists a finite collection of arcs $A_1, A_2,\ldots A_n$ such that for $i\neq j$, $A_i\cap A_j$ is a subset of the endpoints of $A_i$ (and also of $A_j$ by symmetry), with $G=\bigcup A_i$. A graph $G$ has two classes of distinguished points---an \emph{endpoint} of $G$ is a point which is an endpoint of exactly one such arc and a \emph{branchpoint} is a point which is an endpoint of at least three such arcs. Note that every graph is path-connected. 

The collection of spaces in which we are interested are an extension of the collection of graphs in the same sense that the collection of dendrites and dendroids are extensions of the collection of acyclic graphs. Recall that a \emph{dendroid} is an arcwise connected and hereditarily unicoherent continuum (i.e. if $A$ and $B$ are subcontinua with nontrivial intersection, then $A\cap B$ is connected), and that a \emph{dendrite} is a locally connected dendroid (Exercise 10.58 of \cite{Nadler}).  In the 1960's, B. Knaster developed an alternate definition of dendroid--specifically, he defined a dendroid as a continuum $D$ for which for every $\epsilon>0$ there exists a tree $T\subseteq D$ and a retraction $r:D\to T$ that is an $\epsilon$-map (i.e. for all $x\in T$, $\pi^{-1}(x)$ has diameter less than $\epsilon$). While some progress has been made in demonstrating the equivalence of these definitions \cite{Fugate1,Fugate2}, it remains an open problem \cite{OpenProblems2}. It is from this latter notion of dendroid that we draw our analogy.

We call a space $X$ a \emph{graphoid} provided that for all $\epsilon>0$ there exists a graph $G\subseteq X$ and a retraction $\pi_G:X\to G$ which is an \emph{$\epsilon$-map}, i.e. $\pi_G^{-1}(x)$ has diameter less that $\epsilon$ for each $x\in G$. It is worth noting that there may be many such graphs, and for each graph, many such projections. If $X$ is also locally connected, we call $X$ a \emph{graphite}. Note that graphoids and graphites are necessarily one-dimensional, and more specifically \emph{graph-like} (i.e. homeomorphic to an inverse limit of graphs by 2.13 of \cite{Nadler}, the dimension result follows by \cite{Charalambous}) but not every graph-like continuum is a graphoid. 

The class of graphites is indeed much larger than the class of dendrites. Indeed, it includes the Menger curve, which can be seen as follows. Let $M$ denote the Menger curve as a subset of $[0,1]^3$. Note that $M$ retracts onto the edges of the cube $[0,1]^3$ and that this retraction is a $\sqrt{1/3}$-map from $M$ onto a graph contained in $M$. By self-similarity of $M$, we can, for $n\geq 1$, define a $(1/3)^n\sqrt{1/3}$-map from $M$ to a graph contained in $M$ by applying an appropriately scaled version of this retraction to the $20^n$ copies of $M$ with edges of length $(1/3)^n$ which compose $M$. A similar argument demonstrates that the Sierpinski triangle and carpet are also graphites. 

The following useful result holds for all graphoids.

\begin{lemma}\label{approx}
	Let $X$ be a graphoid and $f:X\to X$ a continuous map. For every $\eta>0$ there exists a $\lambda>0$ such that if $G\subseteq X$ is a graph for which $\pi_G:X\to G$ is a $\lambda$-map, and $g:G\to X$ is in $B_{\lambda}(f|_{G})$, then $g\circ \pi_G\in B_\eta(f)$. 
\end{lemma}

\begin{proof}
	Let $X$ be a graphoid and $f:X\to X$. Fix $\eta>0$. By continuity of $f$, there exists $\lambda>0$ such that if $d(x,y)<\lambda$, then $d(f(x),f(y))<\eta/2$. Choose such a $\lambda$ so that $\lambda<\eta/2$.
	
	Now, let $G$ be a graph satisfying the hypotheses of the Lemma. Fix $g:G\to X$ in $B_{\lambda}(f|_G)$. Then for each $x\in X$, we have $d(\pi_G(x),f(x))<\lambda$ and so 
	\[d(g(\pi_G(x)),f(x))\leq d(g(\pi_G(x)),f(\pi_G(x))) +d(f(\pi_G(x)),f(x))<\lambda+\eta/2.\]
	Thus $\sup_{x\in X}\{d(g(\pi_G(x)),f(x))\}\leq \lambda+\eta/2<\eta$ and so $g\circ \pi_G\in B_\eta(f)$ as desired.
\end{proof}

Recall that the \emph{nerve} of an open cover of a space $X$ is the simplicial complex with vertex set equal to $\mathcal U$ and faces given by collections of elements of $\mathcal U$ with common intersection. For an open cover $\mathcal U$ and $U\in\mathcal U$, the \emph{core of $U$ in $\mathcal U$} is defined to be the  set $U\setminus\bigcup_{V\in\mathcal U\setminus \{U\}}\overline V$. It is important to many constructions that each element of an open cover has nonempty core, and as such, it is often  useful to restrict attention to  \emph{taut} open covers, i.e. those open covers $\mathcal U$ of $X$ which satisfy the following conditions:
\begin{enumerate}
	\item if $U,V\in\mathcal U$ and $\overline U\cap \overline V\neq\ns$, then $U\cap V\neq\ns$, and
	\item no proper subset of $\mathcal U$ covers $X$.
\end{enumerate} 
For taut open covers of connected spaces whose nerves are trees it follows immediately that for all $U\in\mathcal U$ which are not endpoints of the nerve, the core of $U$ in $\mathcal U$ is nonempty (c.f. the treatment of chain covers in Chapter 12 of Nadler \cite{Nadler}). A cover in which every core is nonempty is easily constructed from this by adjoining each element of $\mathcal U$ which corresponds to an endpoint with its neighboring element in $\mathcal U$. It is important, however, to note that there exist taut open covers of connected spaces for which each core is empty. A simple example of this consists of the unit circle and the open cover $\mathcal U=\{(0,3\pi/2),(\pi/2,0), (3\pi/2,\pi/2)\}$ where $(a,b)$ denotes the open arc of the circle from $a$ to $b$ going counterclockwise around the circle. 

The following lemma demonstrates that graphites have particularly well structured open covers of arbitrarily small mesh.

\begin{lemma} \label{tautcover}
	Let $X$ be a graphite. For every $\epsilon>0$, there exists a finite open cover  of $X$ consisting of connected sets of diameter less than $\epsilon$ which is taut and whose nerve is a graph containing no three cycles and for which the core of each element is nonempty.
\end{lemma}

\begin{proof}
	
	Fix $\epsilon>0$ and fix $\eta<\epsilon/4$. Since $X$ is a graphite, choose a graph $G\subseteq X$ and a retraction $\pi_G:X\to G$ which is an $\eta$-map. Since $G$ is a graph, it is easy to find an open cover $\mathcal V$ of $G$ consisting of connected sets of diameter less than $\eta$ which is taut and whose nerve is a graph containing no three cycles and for which the core of each element of $\mathcal V$ is nonempty.
	
	Now, define $\mathcal W=\{\pi_g^{-1}(V): V\in\mathcal V\}$ and notice that this is a taut open cover of $X$ with nerve isomorphic to that of $\mathcal V$ and hence has no three cycles. Note that the core of each element of $\mathcal W$ contains the preimage of the core of the corresponding element of $\mathcal V$ and hence is nonempty. Furthermore, if $W\in\mathcal W$ and $x,y\in W$, then $d(x,y)\leq d(x,\pi_G(x))+d(\pi_G(x),\pi_G(y))+d(\pi_G(y),y)<3\epsilon/4$ since $\pi_G(x)$ and $\pi_G(y)$ belong to $\pi_G(W)\in\mathcal V$. However, it is unlikely that the elements of $\mathcal W$ are connected. Since $\mathcal W$ is taut and has no three cycles, we can choose $\gamma>0$ less than $\eta$ such that if $W,W'\in\mathcal W$ and $W\cap W'\neq\ns$, then $d(W,W')>\gamma$.
	
	Towards construction of the claimed cover, for each $W\in \mathcal W$ and $x\in W$, observe that $X$ is locally connected at $x$, and so we can choose a connected open set $N(x,W)$ with $x\in N(x,W)\subseteq W$. Since $X$ is compact, we can choose a finite list of pairs $(x_1,W_1),\ldots, (x_k,W_k)$ in $X\times\mathcal W$ such that $X=\bigcup_{i\leq k}N(x_i,W_i)$. For each $W\in\mathcal W$, define $O(W)$ to be the union of those $N(x_i,W_i)$ which are subsets of $W$. Note that this union contains all sets of the form $N(x_i,W_i)$ with $W_i=W$ but may also contain open sets of the form $N(x_i,W_i)$ with $W_i\neq W$. Notice that for all $W\in\mathcal W$, $W$ has nonempty core, and so there exists $i\leq k$ with $N(x_i,W_i)=N(x_i,W)\subseteq W$. Hence $N(x_i,W_i)\subseteq O(W)\neq\ns$.
	
	For each $W\in\mathcal W$, $O(W)$ is the union of finitely many open connected sets, and therefore has finitely many components, each of which is open. It follows that $\overline{O(W)}$ has finitely many components as well; enumerate them as $C_1(W),\ldots C_{n(W)}(W)$. Fix $\delta>0$ less than $\gamma$ such that if $W,W'\in\mathcal W$, $i\leq n(W)$, $j\leq n(W')$ and $C_i(W)\cap C_j(W')=\ns$, then $d(C_i(W),C_j(W'))>\delta$.

	Now, for each $x\in\partial (O(W))$, choose a connected open neighborhood $B(x,W)$ of $x$ with diameter less than $\delta/2$. Then, for $i\leq n(W)$, define
	\[U_i(W)=C_i(W)\cup\bigcup_{x\in\partial C_i(W)}B(x,W).\]
	
	Observe that for each $i$ and $W$, $U_i(W)$ is open and connected by construction, and is contained in the ball of radius $\gamma$ around $W$. Hence it has diameter less than $\epsilon$, since the diameter of $W$ is less than $3\epsilon/4$ and we have adjoined to it only sets of diameter less than $\epsilon/8$. It also follows that if $C_i(W)\cap C_j(W')=\ns$, then $\overline{U_i(W)}\cap\overline{U_j(W')}=\ns$, hence if $\overline{U_i(W)}\cap\overline{U_j(W')}\neq\ns$, then $U_i(W)\cap U_j(W')\neq\ns$, i.e. the set $\mathcal U=\{U_i(W): W\in\mathcal W, i\leq n(W)\}$ is a taut open cover of $X$ by connected sets of diameter less than $\epsilon$.

	Notice that for a fixed $W\in\mathcal W$, and $i<j\leq n(W)$, we have $U_i(W)\cap U_j(W)=\ns$. Thus, if $U_i(W)$, $U_j(W')$ and $U_k(W'')$ form a three cycle in the nerve of $\mathcal U_0$, then $W$, $W'$ and $W''$ are all distinct and it follows that $W$, $W'$ and $W''$ form a three cycle in $\mathcal W$. Since $\mathcal W$ has no three cycles, it follows that $\mathcal U$ also has no three cycles.
	
	Suppose that $U_i(W)$ has empty core in $\mathcal U$ and $U_i(W)$ is not an endpoint in the nerve of $\mathcal U$. Then there exist $W_1,\ldots, W_m\in\mathcal W$ and indices $j_k\leq n(W_k)$ for $k\leq m$ such that $U_i(W)=\bigcup\left(\overline{U_{j_k}(W_k)}\cap U_i(W)\right)$. But $U_i(W)$ is connected, so there exist $k_0<k_1\leq m$ such that $\overline{U_{j_{k_0}}(W_{k_0})}\cap \overline{U_{j_{k_1}}(W_{k_1})}\neq ns$ and hence ${U_{j_{k_0}}(W_{k_0})}\cap {U_{j_{k_1}}(W_{k_1})}\neq \ns$ . This is a contradiction, as $U_i(W)$, $U_{j_{k_0}}(W_{k_0})$ and $U_{j_{k_1}}(W_{k_1})$ would be a three cycle in the nerve of $\mathcal U_0$.
	
	Thus if $U_i(W)$ has empty core in $\mathcal U$, there exists a unique $W'\in\mathcal W$ and $j\leq n(W')$ such that $U_i(W)\cup U_j(W')$ is connected. Since $U_i(W)\subseteq \overline{U_j(W')}$, the diameter of $U_i(W)\cup U_j(W')$ is equal to that of $U_j(W')$ and thus is less than $\epsilon$.
	
	For each $U_i(W)\in\mathcal U$, define $\hat U_i(W)$ to be the union of $U_i(W)$ and any elements of $\mathcal U_0$ which have empty core and are adjacent to $U_i(W)$. Finally, define
	\[\mathcal {\hat U}=\{\hat U_i(W):W\in\mathcal W', i\leq n(W), \textrm{ and $U_i(W)$ has nonempty core}\}.\]
	
	By construction, $\mathcal {\hat U}$ is an open cover of $X$ by connected sets of diameter less than $\epsilon$. Since $\mathcal {\hat U}$ is derived from $\mathcal U$ in a way that does not introduce any new adjacencies in the nerve of the cover, it is clear that $\mathcal {\hat U}$ inherits both tautness and lack of three cycles from $\mathcal U$, and thus $\mathcal {\hat U}$ is the desired cover of $X$.

\end{proof}

\section{Maps of Graphites} \label{maps}

In this section we prove our main results. The outline of the proof is as follows: fix a graphite $X$ and consider the space $\shad_n(X)\subseteq\cont(X)$ consisting of those maps $f$ such that there exists $\delta>0$ so that every $\delta$-pseudo-orbit is $1/n$-shadowed. A series of lemmas demonstrates that each $\shad_n(X)$ is contains a dense open set in $\cont(X)$. We proceed in a manner similar to that in \cite{BMR-Dendrites}. For every $f\in \cont(X)$ and every $\epsilon>0$, we will construct a map $g\in\cont(X)$ with $\rho(g,f)<\epsilon$ and find a $\gamma>0$ so that $B_\gamma(g)\subseteq B_\epsilon(f)$ and $B_\gamma(g)\subseteq\shad_n(X)$.


\begin{remark}
	For Lemmas \ref{skeleton}, \ref{orbit-filling} and \ref{Tn}, we work with a fixed graphite $X$ and a fixed $f:X\to X$. Also, fix $n\in\N$ and $1/n>\epsilon>0$. Additionally, by Lemma \ref{tautcover}, we can fix a taut open cover $\mathcal U=\{U_1,\ldots U_k\}$  of $X$ by connected sets such that the nerve of the cover is a graph containing no three cycles, the core of each element is nonempty, and such that  for $i\leq k$,  $diam(U_i)$ and $diam(f(U_i))$ are less than $\epsilon/3$.
\end{remark}

Our first goal is to construct a map $g$ near $f$ such that, for each $i\leq k$,   $\phi_{\mathcal U,g}(i)\supseteq \phi_{\mathcal U,f}(i)$ and so that each $\mathcal U$-pseudo-orbit pattern for $g$ is realized by an orbit for $g$. This is accomplished in Lemma \ref{skeleton}. Our second goal is to demonstrate that if $h$ is sufficiently close to $g$, then every $\mathcal U$ pseudo-orbit pattern for $h$ is realized by an orbit for $h$. This is accomplished in Lemma \ref{orbit-filling}. Finally, in Lemma \ref{Tn}, we establish that $g$ and its perturbations belong to $\mathcal T_n$.


\label{key}

\begin{lemma}\label{skeleton}
	There exists a map $g:X\to X$ with $\rho(g,f)<\epsilon$ such that, for each $i\leq k$,  $\phi_{\mathcal U,g}(i)\supseteq \phi_{\mathcal U,f}(i)$  and such that for each sequence $\seq{j_i}_{i\in\omega}$ in $\{1,\ldots, k\}^\omega$ which satisfies $j_{i+1}\in\phi_{\mathcal U, g}(j_i)$, there is a point $x\in X$ with $g^i(x)\in {U_{j_i}}$. Furthermore, if $j\notin\phi_{\mathcal U,g}(i)$ then $\overline{g(U_i)}\cap\overline U_j=\ns$.
\end{lemma}

\begin{proof}
	Let $\mathcal U=\{U_1,U_2,\ldots U_k\}$ be the taut open cover consisting of connected sets as described above and let $\phi_f$ be the pattern of $f$ with respect to $\mathcal U$. For each $i\leq k$, define $V_i$ to be the core of $U_i$ in $\mathcal U$ and recall that by Lemma \ref{tautcover}, this is nonempty and open.
	
	Choose $\eta>0$ such that $\eta<\epsilon$ and such that 
	\begin{enumerate}
		\item if $U_j\cap U_i=\ns$, then $d(\overline{U_i},\overline{U_j})>2\eta$,
		\item for all $i\leq k$, the set $\{x\in U_i:   B_\eta(x)\subseteq V_i\}$ is nonempty, and
		\item for each $i,j\leq k$ with $U_i\cap U_j\neq\ns$, the set $\{x\in X: B_{\eta}(x)\subseteq U_i\cap U_j\}$ is nonempty,
	\end{enumerate}
	
	Now, choose $\lambda>0$ so that $\lambda<\eta/2$ and so that $\lambda$ satisfies Lemma \ref{approx} with respect to $\eta$ and the function $f$,
	and let $\pi_G:X\to G$ be a $\lambda$-map onto a graph $G\subseteq X$.
	
	For $i\in\{1,2,\ldots,k\}$, define \[\phi(i)=\left\{j:{U_j}\cap\bigcup_{l\in\phi_f(i)}{U_l}\neq\ns\right\}.\]
	Notice that for all $i$, $\phi(i)\supseteq\phi_f(i)$ . Our goal is to construct $g$ so that $\phi_g=\phi$. 
	

	The first step in the construction is to define, for each $i\leq k$, a connected subgraph $G_i$ of $G$ such that if $j\in\phi(i)$, then the set $\{x\in G_i: B_\lambda(x)\subseteq V_j\}$ has nonempty interior in $G$ and if $j\notin\phi(i)$, then $d(G_i,\overline U_j)>\eta$. These subgraphs will form the frame on which $g$ will be constructed. Figure \ref{phig} is a schematic of this construction.
	
	
	
	Towards this end, fix $i\leq k$ and $j\in\phi(i)$, and define $D_j:=\pi_G(\overline{U_j})$. By properties (1) and (3) of $\eta$ and the fact that $\pi_G$ is a $\lambda$-map, $D_j$ is a connected, closed subgraph of $G$, and $D_j\cap U_l\neq\ns$ if and only if $U_j\cap U_l\neq\ns$. 
	Now, define $W=\bigcup_{l\notin\phi(i)} B_\eta(\overline{U_l})$, and consider $D_j\setminus W$. Since $j\in\phi(i)$, choose $j'\in\phi_f(i)$ such that $U_j\cap U_{j'}\neq\ns$ and define $C_j$ to be a component of $D_j\setminus W$ which meets $U_{j'}$.  We claim that $G_i=\bigcup_{j\in\phi(i)}C_j$ satisfies the above requirements.
	
	\begin{figure}[h]
		\includegraphics[width=5in]{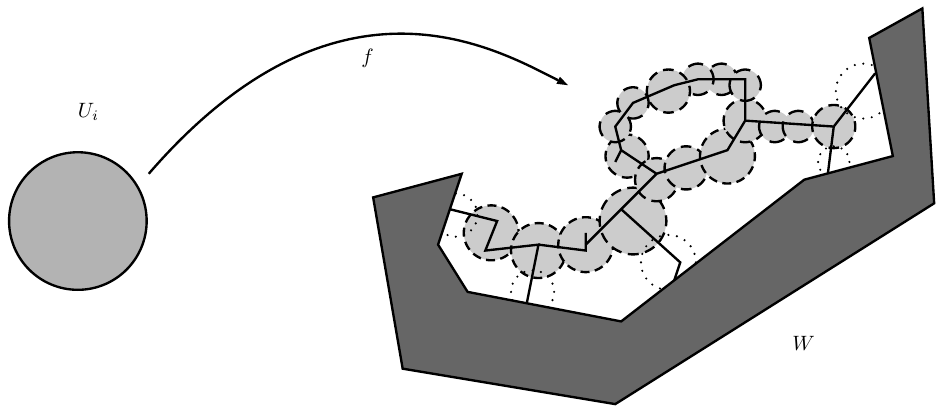}
		\caption{The filled neighborhoods on the right represent those elements of $\mathcal U$ which belong to $\phi_f(i)$, the dotted neighborhoods represent those which belong to $\phi(i)\setminus\phi_f(i)$, and the solid graph in their union is $G_i$. The filled polygonal region represents $W$} \label{phig}
	\end{figure}
	
	Notice that for all $j\in\phi(i)$, we have $C_j\cap W=\ns$, and thus  $d(C_j,\overline{U_l})>\eta$ for all $l\notin \phi(i)$. It follows immediately that  $d(G_i,\overline{U_l})>\eta$ for all $l\notin \phi(i)$.
	
	To verify that $G_i$ is connected, first notice that if $j\in\phi_f(i)$, then for each $l\notin\phi(f)$, $U_j\cap U_l=\ns$, and so $d(\overline{U_j}, \overline{U_l})>2\eta$, and therefore $D_j\cap W=\ns$. It follows then that $C_j=D_j=\pi_G(\overline{U_j})$. Furthermore, $C_j$ meets $\pi_G(f(U_i))$, since $U_j\cap f(U_i)\neq\ns$. If $j\in\phi(i)\setminus\phi_f(i)$, then there exists $j'\in\phi_f(i)$ with $C_j$ meeting $U_{j'}$. Since $\pi_G$ is a retraction, $\pi_G(C_j\cap U_{j'})\subseteq C_j\cap D_{j'}=C_j\cap C_{j'}$. Thus, for each $j\in\phi(i)$, there exists $l\in\phi_f(i)$ such that $C_j$ meets $C_l$. Furthermore, for $l\in\phi(i)$, $C_l$ meets $\pi_G(f(U_i))$, and since all of these sets are connected, their union, $G_i$ is also connected.
	
	Now to verify the final property of $G_i$, that for each $j\in\phi(i)$ the set $\{x\in G_i: B_\lambda(x)\subseteq V_j \}$ must have nonempty interior in $G$. We will demonstrate that for each $j$, the set $\{x\in C_j: B_\lambda(x)\subseteq V_j \}$ has nonempty interior in $C_j$.

	If $j\in\phi(i)$, this is simple since $C_j=D_j=\pi_G(\overline{U_j})$, fix $p\in U_j$ such that $B_\eta(p)\subseteq V_j$ as guaranteed by property (3) of $\eta$. Then $B_{\eta/2-\lambda}(\pi_G(p))\cap C_j$ is a nonempty open subset of $C_j$ which is contained in $B_{\eta/2}(p)$, and for each $x\in B_{\eta/2-\lambda}(\pi_G(p))$, we have $B_\lambda(x)\subseteq B_{\lambda+\eta/2}(p)\subseteq V_j$. 
	
	If $j\notin\phi_f(i)$, then there are two cases to consider. First, if $D_j\cap W=\ns$, then $C_j=D_j$, and as before, the set of points $x$ in $C_j$ for which $B_\lambda(x)\subseteq V_j$ has nonempty interior in $C_j$. If $D_j\cap W\neq\ns$, then $C_j$ is a proper component of $D_j$ which meets $U_{j'}$ for some $j'\in\phi_f(i)$ and meets $\partial W$ by the Boundary Bumping Theorem. Let $p\in C_j\cap\partial W$ and fix $w\leq k$ such that $w\notin\phi(i)$ and $d(p,\overline{U_w})=\eta$. Notice that this implies that $U_w\cap U_j\neq\ns$. Then for $l\notin\phi(i)$, $d(p,\overline{U_l})\geq\eta$ by construction of $W$. For $l\in\phi(i)$, if $l\neq j$ and $d(p,\overline{U_l})<\eta$, then both $d(\overline{U_l},\overline{U_j})$ and $d(\overline{U_l},\overline{U_w})$ are less than $2\eta$, and therefore $U_l\cap U_j$ and $U_l\cap U_w$ are both nonempty. But this implies that $U_j$, $U_l$ and $U_w$ form a three cycle in the nerve of $\mathcal U$, a contradiction. Thus, for all $l\neq j$, $B_\eta(p)\cap \overline U_j=\ns$, and therefore, for all $x\in B_\lambda(p)\cap C_j$, we have $B_\lambda(x)\subseteq V_j$.

	Having verified that the graphs $G_i$ have the specified properties, we will now select a family of \emph{free arcs} in $G$ (i.e. arcs which do not contain endpoints or branchpoints of $G$) which will be stretched across the subgraphs $G_i$ to construct the map $g$. For $i\in\phi(p)$, the set $\{x\in G_p: B_\lambda(x)\subseteq V_i\}$ has nonempty interior in $G$, so we can define $L_{i,p}\subseteq\{x\in G_p: B_\lambda(x)\subseteq V_i\}$ to be a free arc in $G$. Similarly, for all $i\leq k$,  we can choose an arc $L_{i,0}\subseteq U_i\cap G$ which is free in $G$ and such that $B_\lambda(L_{i,0})\subseteq V_i$ (the arc $L_{i,0}$ is not strictly necessary if $i$ is in the range of $\phi$).
	
	Since there are only finitely many such arcs to choose, we can, by shrinking them if necessary, choose these arcs so that the family $\{L_{i,p}:i\in\phi(p), p\leq k\}\cup\{L_{i,0}:i\leq k\}$ is pairwise disjoint. It is important to point out that for each pair $i,p$, $B_\lambda(L_{i,p})\subseteq V_i$, and thus if $d(L_{j,p},U_i)<\lambda$, then $i=j$. In particular, if $x\in U_i$ and $\pi_G(x)\in L_{j,p}$, then $j=i$.
	
	Now, we begin construction of $g$. For each pair $i,p\leq k$ with $p=0$ or with $i\in\phi(p)$ , we begin by defining a function $g_{i,p}:L_{i,p}\to G_i$. Fix $b_{i,p}(0)$ and $a_{i,p}(k+1)$ to be the endpoints of $L_{i,p}$ and order this arc so that $b_{i,p}(0)<a_{i,p}(k+1)$. Now, for $0<j\leq k$ choose points $a_{i,p}(j)$ and $b_{i,p}(j)$ in $L_{i,p}$ such that \[b_{i,p}(0) < a_{i,p}(1)<b_{i,p}(1)<a_{i,p}(2)<\cdots<a_{i,p}(k)<b_{i,p}(k)<a_{i,p}(k+1).\]
	
	The function $g_{i,p}$ is then defined as follows (See Figure \ref{gip}):
	\begin{enumerate}
		\item $g_{i,p}(b_{i,p}(0))=\pi_G(f(b_{i,p}(0)))$,
		\item $g_{i,p}(a_{i,p}(k+1))=\pi_G(f(a_{i,p}(k+1)))$,
		\item for $j\in\phi(i)$, $g_{i,p}$ maps the subarc $[a_{i,p}(j),b_{i,p}(j)]$ via a homeomorphism onto the arc $L_{j,i}$, and
		\item for each component $(b_{i,p}(j),a_{i,p}(l))$ on which $g_{i,p}$ is not already defined, define $g_{i,p}$ to map $[b_{i,p}(j),a_{i,p}(l)]$ via homeomorphism onto a subarc of $G_i$ from $g_{i,p}(b_{i,p}(j))$ to $g_{i,p}(a_{i,p}(l))$.
	\end{enumerate}

	\begin{figure}[h] 
		\includegraphics[width=4in]{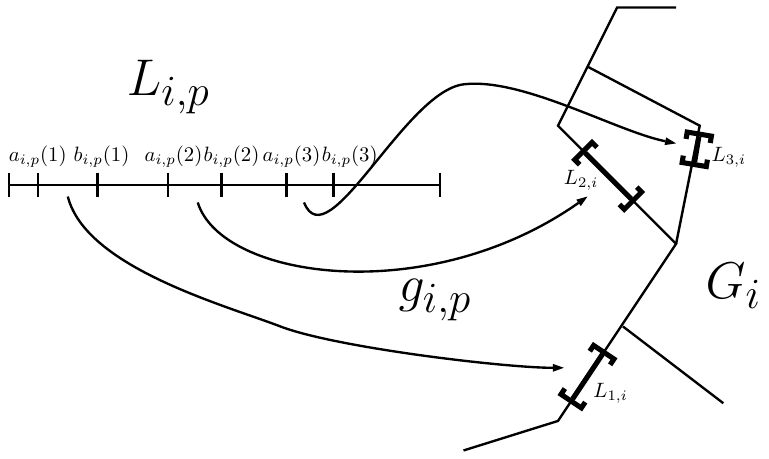}
		\caption{The map $g_{i,p}$}\label{gip}
	\end{figure}

	Finally, define $g:X\to X$ by
	\[g(x)=\begin{cases} 
	\pi_G(f(\pi_G(x))) & \pi_G(x)\notin\bigcup L_{i,p} \\
	g_{i,p}(\pi_G(x)) & \pi_G(x)\in {L_{i,p}}
	\end{cases}\]
	
	Note that $g$ is continuous by construction. All that remains is to establish that $g$ satisfies the claimed properties.
	
	First, we will demonstrate that $\rho(f,g)<\epsilon$. For $x\in X$, there are two cases; either $\pi_G(x)\notin\bigcup L_{i,p}$ or $\pi_G(x)\in L_{i,p}$ for some pair $i,p\leq k$ with $p=0$ or $i\in\phi(p)$.
	
	For $x\in X$ with $\pi_G(x)\notin\bigcup L_{i,p}$, this follows from Lemma \ref{approx}. Since $\pi_G$ is a $\lambda$-map and $\lambda$ satisfies the hypotheses of Lemma \ref{approx} for the function $f$, and since $\pi_G\circ f|G$ is in $B_\lambda(f|G)$, then $\pi_G\circ f\circ\pi_G\in B_\eta(f)$. Finally, since $\eta<\epsilon$, if $\pi_G(x)\notin\bigcup_{i\in\phi(p)}L_{i,p}$, we have $g(x)=\pi_G\circ f\circ\pi_G(x)$ is within $\epsilon$ of $f(x)$. 
	
	Now, fix $i,p\leq k$ with $p=0$ or $i\in\phi(p)$. If $x\in X$ with $\pi_G(x)\in L_{i,p}$, then $d(x,L_{i,p})<\lambda$, and so $x\in B_\lambda(L_{i,p})\subseteq U_i$, and thus $f(x)\in f(U_i)\subseteq \bigcup_{l\in\phi_f(i)} U_l\subseteq \bigcup_{l\in\phi(i)} U_l$. But $g(x)\in g(L_{i,p})\subseteq G_i\subseteq\bigcup_{l\in\phi(i)}$, and so $d(f(x),g(x))<diam(\bigcup_{l\in\phi(i)}U_l)$. Now, for each $l\in\phi(i)$, there exists $j\in\phi_f(i)$ with $U_l\cap U_j\neq \ns$ and $U_j\cap f(U_i)\neq\ns$, and so
	\[diam\left(\bigcup_{l\in\phi(i)}U_l\right)<diam(f(U_i))+2\max\{diam(U_l):l\leq k\}<\epsilon.\] 
	Thus $d(f(x),g(x))<\epsilon$ as desired.
	
	Next, we will demonstrate that $\phi_g=\phi$, and hence $\phi_g\supseteq\phi_f$. Fix $i\leq k$ and $j\in\phi(i)$ and observe that there exists $p\leq k$ (possibly equal to zero) such that $L_{i,p}\subseteq U_i$. By construction, $g(L_{i,p})$ contains $L_{j,i}$. But $L_{j,i}\subseteq U_j$, and so $g(U_i)\cap U_j\neq\ns$, i.e $j\in\phi_g(i)$. Now, consider $i\leq k$ and $j\in\phi_g(i)$ and let $x\in U_i$ witness this. Then $g(x)\in U_j$. If $\pi_G(x)\in  L_{j,p}$ for some $j\leq k$, then by construction of these arcs, $j=i$, and thus $g(x)\in G_i$, and hence $U_j\cap G_i\neq\ns$, in which case $j\in\phi(i)$. Otherwise, $\pi_G(x)\notin\bigcup L_{i,p}$, in which case, $g(x)=\pi_G\circ f\circ\pi_G(x)\in B_\eta(f(x))$ by choice of $\lambda$ and Lemma \ref{approx}. But $f(x)\in U_l$ for some $l\in\phi_f(i)$, and hence $f(\overline{U_l},\overline{U_j})<\eta$. By choice of $\eta$, this implies that $U_j$ meets $U_l$ and since $l\in\phi_f(i)$, we have $j\in\phi(i)$. Thus $\phi_g(i)=\phi(i)$ as claimed, and since $\phi(i)\supseteq \phi_f(i)$ by construction, we have that $\phi_g(i)\supseteq\phi_f(i)$.

	Now, let $\seq{j_i}_{i\in\omega}$ be a sequence satisfying $j_{i+1}\in\phi_g(j_i)$ for all $i\in\omega$ we will demonstrate that there exists $x\in X$ such that for all $i\in\omega$, $g^i(x)\in U_{j_i}$. For simplicity of notation, define $j_{-1}=0$. Then, by construction, for each $i\in\omega$, $g(L_{j_{i},j_{i-1}})$ contains $L_{j_{i+1},j_i}$ and so, by compactness, the set $\bigcap_{i\in\omega}g^{-i}(L_{j_i,j_{i-1}})$ is nonempty.
	Now, choose $x\in \bigcap_{i\in\omega}g^{-i}(L_{j_i,j_{i-1}})$,  and notice that for all $i\in\omega$,  $g^i(x)\in L_{j_{i},j_{i-1}}\subseteq U_{j_i}$ as claimed.
	
	Finally, suppose that $x\in U_i$, and $j\notin \phi_g(i)$. If $\pi_G(x)\in L_{i,p}$ for some $p\leq k$, in which case $g(x)\in G_i$, but $d(G_i,\overline U_l)\geq\eta$ by construction. Else, $\pi_G(x)\notin L_{j,p}$ for any pair $j,p\leq k$. In this case $d(g(x),f(x))<\eta$ by Lemma \ref{approx} and thus $d(g(x),U_l)<\eta$ for some $l\in\phi_f(i)$. If $d(g(x),U_j)<\eta$, then $d(U_j,U_l)<2\eta$, and by property (1) of $\eta$, it follows that $U_j\cap U_l\neq\ns$, and therefore $j\in\phi_g(i)$, a contradiction. Thus $d(\overline{g(U_i)},\overline {U_j})\geq\eta$, and in particular, $\overline{g(U_i)}\cap\overline {U_j}=\ns$ as claimed.
	
\end{proof}

\begin{lemma}\label{orbit-filling}
	Let $g$, $\eta$, $\lambda$, $G$, $\pi_G$ and $\{L_{i,p}\}$ be as defined in Lemma \ref{skeleton}. There exists $\gamma>0$ such that for all maps $h:X\to X$ with $\rho(h,g)<\gamma$, $\phi_{\mathcal U,g}=\phi_{\mathcal U, h}$ and for each sequence $\seq{j_i}_{i\in\omega}$ in $\{1,\ldots, k\}^\omega$ which satisfies $j_{i+1}\in\phi_{\mathcal U, h}(j_i)$, there is a point $x\in X$ with $h^i(x)\in {U_{j_i}}$.
\end{lemma}

\begin{proof}
	First, observe that for each $i$, $\overline{g({U_i})}$ meets only those $\overline{U_j}$ with $j\in\phi_{g}(i)$. In particular, then, for each $i$, there exists $\tau_i>0$ such that $d(g(\overline{U_i})\cap \overline{U_j})<{\tau_i}$ if and only if $j\in\phi_{g}(i)$. Taking $\tau$ to be the minimum of the $\tau_i$, we observe that if $h:X\to X$ satisfies $\rho(g,h)<\tau$, then $h(\overline U_i)\subseteq \bigcup_{j\in\phi_{g}(i)}\overline U_j\setminus \bigcup_{j\notin\phi_{g}(i)}\overline U_j$, and thus $\phi_{h}(i)\subseteq\phi_{g}(i)$.
	
	Now, we need only determine a tolerance which assures the other inclusion. Towards this end, choose $\xi>0$ such that, for all $i,p\leq k$ with $p=0$ or $i\in\phi_g(p)$, we have $B_\xi(L_{i,p})\cap G$ is an arc (this is possible since each $L_{i,p}$ is free in $G$) and such that \[\xi<\min\{d(b_{i,p}(0),[a_{i,p}(1), b_{i,p}(k)]), d(a_{i,p}(k+1),[a_{i,p}(1), b_{i,p}(k)])\}.\] In particular, if $b,a\in G$ with both $d(b, b_{i,p}(0))$ and $d(a,a_{i,p}(k+1))$ less than $\xi$, then the arc from $b$ to $a$ contained in $B_\xi(L_{i,p})\cap G$ contains $[a_{i,p}(1), b_{i,p}(k)]$.
	
	Having chosen $\xi$, we choose $\gamma>0$ such that 
	\begin{enumerate}
		\item $\gamma<\tau$,
		\item $\gamma<\lambda$,
		\item $\gamma<\xi$,
		\item if $d(x,y)<\gamma$, then $d(\pi_G(x),\pi_G(y))<\xi$, and
		\item if $d(x,y)<\gamma$, then $d(g(x),g(y))<\xi$.
	\end{enumerate}  We claim that this $\gamma$ is the necessary tolerance.
	
	Let $h:X\to X$ with $\rho(g,h)<\gamma$. By (1) and the first paragraph above, $\phi_{h}(i)\subseteq\phi_{g}(i)$. Now, choose $p\leq k$ such that $L_{i,p}\subseteq U_i$ is defined. Let $j\in\phi_{g}(i)$, and consider $x\in [b_{i,p}(j),a_{i,p}(j)]\subseteq L_{i,p}\subseteq G$. Observe that by (2) $h(x)\in B_\gamma(g(x))\subseteq B_\lambda (L_{j,i})\subseteq V_j\subseteq U_j$ by choice of $\gamma$, and so $j\in\phi_{h}(i)$. This establishes that $\phi_{g}=\phi_{h}$.
	
	Now, let $\seq{j_i}_{i\in\omega}$ in $\{1,\ldots, k\}^\omega$ be a sequence which satisfies $j_{i+1}\in\phi_{h}(j_i)$. We will inductively demonstrate the existence of the desired point by constructing a nested sequence of arcs, each of which witnesses the desired pattern for a finite (but increasing) length. 
	
	First, let $I_0=[a_{j_0,0}(j_1),b_{j_0,0}(j_1)]\subseteq L_{j_0,0}$. Since $d(h(a_{j_0,0}), g(a_{j_0,0}))<\gamma$, we have $d(\pi_G(h(a_{j_0,0})), \pi_G(g(a_{j_0,0}))<\xi$, but we also have $\pi_G(g(a_{j_0,0}(j_1)))=\pi_G(b_{j_1,j_0}(0))=b_{j_1,j_0}(0)$, so $d(h(a_{j_0,0}(j_1)),b_{j_1,j_0}(0))<\xi$. Similarly, we have $d(h(b_{j_0,0}(j_1)),a_{j_1,j_0}(k+1))<\xi$. Thus $\pi_G(h(I_0))$ is a connected subset of $B_\xi(L_{j_1,j_0})\cap G$ which contains points within $\xi$ of both endpoints of $L_{j_1,j_0}$, and thus by choice of $\xi$, $\pi_G(h(I_0))\supseteq [a_{j_1,j_0}(1),b_{j_1,j_0}(k)]$. Since $\pi_G\circ h$ is a continuous map, there exists an arc $I_1\subseteq I_0\subseteq U_{j_0}$ with $\pi_G(h(I_1))=[a_{j_1,j_0}(1),a_{j_1,j_0}(k)]$.
	
	Now, suppose that $I_n$ is a subarc of $I_{n-1}$ such that
	\begin{enumerate}
		\item $h^n(I_n)\subseteq B_\gamma(L_{j_n, j_{n-1}})$,
		\item $\pi_G(h^{n}(I_n))=[a_{j_{n}, j_{n-1}}(1),b_{j_{n}, j_{n-1}}(k)]\subseteq L_{j_n,j_{n-1}}$, and
		\item $h^{n-1}(I_n)\subseteq U_{j_{n-1}}$.
	\end{enumerate}    
	
	Now, since $\pi_G\circ h$ is continuous, we can fix $a<b\in L_n$ such that
	\begin{enumerate}
		\item $\pi_G(h^n(a))=a_{j_{n}, j_{n-1}}(j_{n+1})$,
		\item  $\pi_G(h^n(b))=b_{j_{n}, j_{n-1}}(j_{n+1})$, and
		\item  $\pi_G(h^n([a,b]))=[a_{j_{n}, j_{n-1}}(j_{n+1}),b_{j_{n}, j_{n-1}}(j_{n+1})]$.
	\end{enumerate}   
	
	Now, for all $x\in[a,b]$, $h^{n+1}(x)$ is within $\gamma$ of $L_{j_{n+1}, j_n}$ because we have $d(h^{n+1}(x),g(h^n(x))<\gamma$ and $g(h^n(x))=g(\pi_G(h^n(x)))$ belongs to the set $g([a_{j_{n}, j_{n-1}}(j_{n+1}),b_{j_{n}, j_{n-1}}({j_{n+1}})])=L_{j_{n+1}, j_n}$. Thus $h^{n+1}([a,b])\subseteq B_\gamma(L_{j_{n+1},j_n})$.
	
	Since $\pi_G$ is a $\lambda$-map and a retraction, we see $ h^n([a,b])\subseteq B_\lambda(L_{j_n,j_{n-1}})\subseteq U_{j_n}$.
	
	Also, $h^{n+1}(a)=h(h^n(a))$ is within $\gamma$ of $g(h^n(a))=g(\pi_G(h^n(a)))=g(a_{j_{n},j_{n-1}}({j_{n+1}}))=b_{j_{n+1}, j_n}(0)$ and thus we see that $\pi_G(h^{n+1}(a))$ is within $\xi$ of $\pi_G(b_{j_{n+1}, j_n}(0))=b_{j_{n+1}, j_n}(0)$. Similarly, $\pi_G(h^{n+1}(b))$ is within $\xi$ of $b_{j_{n+1}, j_n}(k+1)$. 
	
	Thus $\pi_G(h^{n+1}([a,b]))$ is a connected subset of $B_\xi(L_{j_{n+1},j_n})\cap G$ which contains points within $\xi$ of both endpoints of $L_{j_{n+1},j_n}$, and so $\pi_G(h^{n+1}([a,b])) \supseteq [a_{j_{n+1},j_n}(1),b_{j_{n+1},j_n}(k+1)]$. Thus, by continuity, we can choose a subarc $I_{n+1}$ of $[a,b]$ with $\pi_G(h^{n+1}(I_{n+1})) = [a_{j_{n+1},j_n}(1),b_{j_{n+1},j_n}(k+1)]$. Since $I_{n+1}\subseteq[a,b]$, we also have $h^{n+1}(I_{n+1})\subseteq B_\gamma(L_{j_{n+1},j_n})$ and $h^n(I_{n+1})\subseteq U_{j_n}$.

	By construction, we see that if $x\in I_n$, then $h^i(x)\in U_{j_i}$ for all $i<n$. By compactness $\bigcap_{n\in\omega}I_n\neq\ns$, and for any $x\in \bigcap_{n\in\omega}I_n$, we have $h^i(x)\in{U_{j_i}}$. Since the sequence $\seq {j_i}$ chosen was arbitrarily, the function $h$ has the desired property.

\end{proof}

\begin{lemma}\label{Tn}
	For $g:X\to X$ as in Lemma \ref{skeleton} and $\gamma>0$ as in Lemma \ref{orbit-filling}, $B_\gamma(g)\subseteq \mathcal T_n$.
\end{lemma}

\begin{proof}
	Let $h\in B_\gamma(g)$ and choose $\delta>0$ to be the Lebesgue number for $\mathcal U$. Let $\seq{x_i}_{i\in\omega}$ be a $\delta$-pseudo-orbit. Then for all $i\in\omega$ with $i>0$, the set $\{x_i,f(x_{i-1})\}$ has diameter less than $\delta$< and so we can fix $j_i\leq k$ such that $x_{i},f(x_{i-1})\in U_{j_i}$. Additionally, fix $j_0$ so that $x_0\in U_0$.  
	
	Then, for each $i\in\omega$, $j_{i+1}\in\phi_h(j_i)$, and so by Lemma \ref{orbit-filling}, there exists $x\in X$ such that for each $i\in\omega$, $h^i(x)\in {U_{j_i}}$. Then $x_i,h^i(x)\in U_{j_i}$, and so $d(x_i,h^i(x))<\textrm{diam}(U_{j_i})<\epsilon/3<1/n$, i.e. the point $x$ $1/n$-shadows $\seq{x_i}$.
	
	Thus $h\in \mathcal T_n$ as claimed.
\end{proof}

We are now ready to prove our main theorem. It is worth recalling that in the preceding sequence of lemmas, $X$ is assumed to be a graphite, and $f\in\cont(X)$.

\begin{theorem}
	Let $X$ be a graphite. Then $\shad(X)$ contains a dense $G_\delta$ subset of $\cont(X)$.
\end{theorem}

\begin{proof}
	By Lemmas \ref{skeleton}, \ref{orbit-filling}, and \ref{Tn}, for each $f\in\cont(X)$ and each $\epsilon>0$ there exists a map $g\in B_\epsilon(f)$ and $\gamma>0$ such that $B_\gamma(g)\subseteq\mathcal T_n$. In other words, $\mathcal T_n$ contains a dense open set. As $\shad(X)$ contains $\bigcap_{n\in\N}\mathcal T_n$, we have proven our claim.
\end{proof}

\section{Surjections of Graphites} \label{surj}

In fact, this process is a fair bit more robust than necessary in the preceding argument, and we can make a minor alteration of the proof of Lemma \ref{orbit-filling} that will allow us to demonstrate that shadowing is generic amongst surjections of graphites. 
Here we make use of the fact that since a graphite is a locally connected one-dimensional continuum, it is necessarily Peano. As such,  for all $\epsilon>0$, there exists a finite collection of Peano continua of diameter less than $\epsilon$, the union of which covers $X$ (Theorem 8.10 of \cite{Nadler}). We will also make use of the fact that any connected open subset of a Peano continuum is arc-connected (Theorem 8.26 of \cite{Nadler}).

\begin{theorem} \label{surjections}
	Let $X$ be a graphite. Then $\shad(X)$ contains a dense $G_\delta$ subset of $\surj(X)$.
\end{theorem}

\begin{proof}
	
	Let $f$ be a surjection, and fix $n\in\N$ and $1/n>\epsilon>0$. We begin by choosing $\mathcal U$ as in the discussion following Lemma \ref{approx}, i.e. $\mathcal U$ is a finite taut open cover by connected sets with the diameters of each $U_i$ and $f(U_i)$ less than $\epsilon/3$ and whose nerve contains no three cycles. Now, by the aforementioned result from \cite{Nadler}, choose a collection $P_1,\ldots P_m$ of Peano continua which cover $X$ and have the properties that for each $P_i$, there exists $U_j\in\mathcal U$ with $P_i\subseteq U_j$ and if $U_l\cap U_j=\ns$ and $P_i\cap U_l\neq\ns$ then $d(P_i,U_j)>\epsilon$.  Let $\psi:\{1,\ldots, k\}\to 2^{\{1,\ldots, m\}}$ be the function defined by $j\in\psi(i)$ if $P_j\cap f(U_i)\neq\ns$. Note that since $f$ is a surjection, $\psi$ is surjective in the sense that for each $1\leq j\leq m$, there exists $i$ with $j\in\psi(i)$.
	
	
	We now proceed in the manner of the proof of Lemma \ref{skeleton} until we define the maps $g_i$.
	
	For each $j\in\psi(i)$, define a subcontinuum $Q_{j,i}$ of $X$ as follows. Since $P_j$ meets $f(U_i)$, there exists $l\in\phi_f(i)$ with $U_l\cap P_j$. Since $U_l$ is arc-connected, we can choose an arc $A_{j,i}\subseteq U_l$ with one endpoint in $G_i$ and the other in $P_j$. Define $Q_{j,i}=P_j\cup A_{j,i}$. Notice that if $j\in\psi(i)$, then $P_j\subseteq\bigcup_{l\in\phi(i)}U_l$ and $d(P_j,U_l)>\epsilon$ for all $l\notin\phi(i)$. It then follows that $Q_{j,i}\subseteq\bigcup_{l\in\phi(i)} U_l$ and if $l\notin \phi(i)$, then $d(Q_{j,i},U_l)>\epsilon$.
	
	For each pair $i,p\leq k$ with $i\in\phi(p)$ or $p=0$, we choose a collection of points (this is a larger collection than in Lemma \ref{skeleton}) $b_{i,p}(0)<a_{i,p}(1)<b_{i,p}(1)<\cdots<a_{i,p}(k)<b_{i,p}(k)<a_{i,p}(k+1)<\cdots<a_{i,p}(k+m)<b_{i,p}(k+m)<a_{i,p}(k+m+1)$ in $I_{i,p}$.
	
	Then, to define $g_{i,p}:I_{i,p}\to X$ with these additional points, we split step (3) into steps (3a) and (3b) as indicated.
	
	\begin{enumerate}
		\item[(3a)] for $1\leq j\leq k$, if $j\in\phi(i)$, $g_{i,p}$ maps the subarc $[a_{i,p}(j),b_{i,p}(j)]$ via homeomorphism onto the arc $I_{j,i}$;
		\item[(3b)]  for $k+1\leq j\leq k+m$, if $j-k\in\psi(i)$, $g_{i,p}$ maps $[a_{i,p}(j),b_{i,p}(j)]$ onto $Q_{j,i}$ as follows. Since $Q_{j,i}$ is Peano  and meets $G_i$ we can, by the Hahn-Mazurkiewicz Theorem (Theorem 8.14 of \cite{Nadler}), define $g_{i,p}:[a_{i,p}(j),b_{i,p}(j)]\to Q_{j,i}$ to be a continuous surjection which maps both $a_{i,p(j)}$ and $b_{i,p}(j)$ to a point in $G_i$.
	\end{enumerate}
	
	The construction of the map $g$ then carries on as indicated in the proof of Lemma \ref{skeleton}.
	Note that this additional step guarantees that if $f$ is surjective, then $g$ is surjective as well (as each $P_j$ is covered at least once). Additionally, if $j\in\psi(i)$, the image of $[a_{i,p}({k+j}),b_{i,p}({k+j})]$ is equal to $Q_{j,i}$, which is a subset of $\bigcup_{l\in\phi(i)} U_l$.  Recall that if $l$ is not in $\phi(i)$,  $d(U_l,Q_{j,i})>\epsilon$. This guarantees that we maintain that $\phi=\phi_{g}$, and thus the map $g$ satisfies all the properties specified and is also a surjection.
	
	By the same argumentation as Lemma \ref{orbit-filling}, any map $h$ within $\gamma$ of $g$ also satisfies the properties specified in Lemma \ref{orbit-filling}, i.e. $B_\gamma(g)\subseteq \mathcal T_n$.
	
	Thus $B_\gamma(g)\cap \surj(X)$ is an open subset of $\surj(X)$ which belongs to $\mathcal T_n$. Thus $\mathcal T_n\cap\surj(X)$ is an open dense subset of $\surj(X)$, and by the same argumentation as in Lemma \ref{Tn}, we see that $\shad(X)$ contains a dense $G_\delta$ subset of $\surj(X)$.
\end{proof}

As a corollary to Theorem \ref{surjections}, we have the following. This result is an extension of those found in \cite{BMR-Dendrites}.

\begin{corollary}
	Let $X$ be a dendrite. Then $\shad(X)$ contains a dense $G_\delta$ subset of $\surj(X)$.
\end{corollary}

\begin{proof}
	By comments in Section 2, every dendrite is a graphite. Applying Theorem \ref{surjections} completes the proof.
\end{proof}

\section{Conclusion} \label{conc}

As it contains, among other things, the Menger continuum, the class of graphites is indeed more general than the previously studied class of dendrites in \cite{BMR-Dendrites}. However, there are no known examples of locally connected one-dimensional continua which are \emph{not} graphites. This leaves the following open question.

\begin{question}
	Is every locally connected one-dimensional continuum a graphite?
\end{question}

Indeed, this question seems intricately related to the open problem concerning the characterization of dendroids mentioned in Section \ref{Prelim}. As such, answers to this question are naturally correlated with analogous questions for dendroids \cite{Eb-Fugate, Fugate1, Fugate2, Marsh-Prajs, Martinez}.

This question also leads naturally to the following notion. As mentioned in the introduction, there have been many results regarding the genericity of shadowing in the classes $\cont(X), \surj(X)$, and $\homeo(X)$ for various types of spaces $X$. A common thread amongst these results is that all of the spaces $X$ for which such results exist are locally connected continua. This leads to the following conjecture. A positive answer to the previous question would provide additional evidence.

\begin{conjecture}
	Let $X$ be a locally connected continuum. Then $\shad(X)$ is generic in $\cont(X)$ and $\surj(X)$. 
\end{conjecture}

In the class of homeomorphisms, the situation is less clear. This is in part due to the fact that locally connected continua, by nature of their heterogeneity, generally have a far sparser class of homeomorphisms than do the manifolds on which shadowing has been demonstrated to be generic in the class of homeomorphisms \cite{Pilyugin-Plam}.

It should be noted that demonstrating that shadowing is not generic in the class of continuous functions is rather difficult and has not yet been the subject of much study. This leads to the following final question.

\begin{question}
	Does there exist a continuum $X$ for which $\shad(X)$ is not generic in $\cont(X)$?
\end{question}


\end{document}